\documentclass[12pt]{amsart}
\usepackage[utf8]{inputenc}
\usepackage{verbatim}
\usepackage{lmodern}
\usepackage{geometry}
\usepackage{amsthm,amsmath,amscd,amssymb}
\usepackage[mathlines,pagewise]{lineno}
\usepackage{cite}
\usepackage{hyperref}
\hypersetup{
    colorlinks,
    citecolor=black,
    filecolor=black,
    linkcolor=black,
    urlcolor=black
}

\DeclareMathOperator{\Hom}{Hom}
\DeclareMathOperator{\Ext}{Ext}

\newtheorem{theorem}{Theorem}[section]
\newtheorem*{theorem*}{Theorem}
\newtheorem{corollary}[theorem]{Corollary}
\newtheorem{lemma}[theorem]{Lemma}
\newtheorem{proposition}[theorem]{Proposition}
\newtheorem*{definition}{Definition}

\title[Restriction of the oscillator representation to dual pairs of type I]{Projective cases for the restriction of the oscillator representation to dual pairs of type I}

\author{Sabine J. Lang}
\address{Department of Mathematics\\ University of Utah\\ Salt Lake City, UT 84112}
\email{lang@math.utah.edu}
\begin{document}

\begin{abstract}
For all the irreducible dual pairs of type I $(G,G')$, we analyze the restriction of the oscillator representation as a $(\mathfrak{g}', K')$-module, when $G'$ is the smaller group. For all $(G,G')$ in the stable range, as well as one more case, the modules obtained are projective. We use the duality correspondence introduced by Howe to analyze these restrictions.
\end{abstract}

\maketitle

\section{Introduction}

A classical problem in representation theory is the understanding of the restriction of a representation $\Pi$ of a group $G$ to one of its subgroups $H$. This work will focus on $(\mathfrak{g},K)$-modules, as defined by Harish-Chandra. In that setting, if $\Pi$ is a $(\mathfrak{g},K)$-module, it is useful to analyze $\Hom_{(\mathfrak{h},H\cap K)}(\Pi,\pi)$, where $\pi$ is an $(\mathfrak{h},H\cap K)$-module. For this purpose, one may use the derived functors: calculating $\Ext^n_{(\mathfrak{h},H\cap K)}(\Pi, \pi)$ is not necessarily easier than $\Hom_{(\mathfrak{h},H\cap K)}(\Pi,\pi)$, but their Euler characteristic might be. 
This difficult part becomes much simpler when the restriction of $\Pi$ is a projective ${(\mathfrak{h},H\cap K)}$-module. In this case, $\Ext^n_{(\mathfrak{h},H\cap K)}(\Pi, \pi)$ vanishes for every $n > 0$. It motivates this paper: the projectivity of a representation is an extremely powerful property. The link between Euler characteristic and projectivity is emphasized in \cite{savin}, among others.

We focus on dual pairs, an approach introduced in the framework of the duality correspondence for the oscillator representation. A dual pair is a pair $(G,G')$ of subgroups of a symplectic group $Sp(V)$, such that $G$ is the centralizer of $G'$ in $Sp(V)$, and vice-versa. This work focuses on dual pairs of type I and uses the Fock model of the oscillator representation, $\omega$. We prove:

\begin{theorem*}
Let $G'$ be the smaller member of a dual pair $(G,G')$ in a symplectic group $Sp(V)$. Then the restriction of the Fock model of the oscillator representation $\omega$ of $\widetilde{Sp}(V)$ to $G'$ is a projective $(\mathfrak{g}',K')$-module under the condition (*), as listed in theorem \ref{MainResult}. This condition includes the stable range but is slightly less restrictive.
\end{theorem*}

It might seem unusual to focus on only one representation of one group. Due to the importance of the oscillator representation, this is however not surprising. This representation appears as (Segal-Shale)-Weil representation (see \cite{segal}, \cite{shale}, \cite{weil}), harmonic or metaplectic representation, among many other names. The theory of duality correspondence (or Theta correspondence) describes the representations that appear in the decomposition of the oscillator representation after restriction to a dual pair, see \cite{howeTransc} or \cite{vergne} for more details. This is one of the major tool used in this work.

\subsection*{Acknowledgment}
I would like to thank my advisor Gordan Savin, without whom none of this work would be possible. I am also grateful to Peter Trapa for his useful comments and advice on this paper. This work was partially supported by the National Science Foundation under Grant DMS-1901745.

\section{Generalities}
Let $G$ be a Lie group with complexified Lie algebra $\mathfrak{g}$, and let $K$ be a maximal compact subgroup in $G$, or its two-fold cover (as needed). We denote by $\mathfrak{k}$ the complexified Lie algebra of $K$, and we choose a Cartan subalgebra $\mathfrak{t}$ of $\mathfrak{g}$ contained in $\mathfrak{k}$.

\subsection{Highest weight modules}

We have the Cartan decomposition $\mathfrak{g} = \mathfrak{k}+\mathfrak{p}$. When the Lie algebra $\mathfrak{k}$ has a non-trivial center but $\mathfrak{g}$ is simple, $\mathfrak{k}$ acts non-trivially on $\mathfrak{p}$. This action decomposes $\mathfrak{p}$ into two weight spaces, $\mathfrak{p}_+$ and $\mathfrak{p}_-$. We let $\Delta$ be the set of roots of $\mathfrak{g}$ with respect to $\mathfrak{t}$. Let $\mathfrak{b}_{\mathfrak{k}}$ is a Borel subalgebra for $\mathfrak{k}$ containing $\mathfrak{t}$. By choosing a Borel subalgebra defined by $\mathfrak{b}=\mathfrak{b}_{\mathfrak{k}}+\mathfrak{p}_+$, we guarantee that the roots of $\mathfrak{p}_+$ are contained in the positive roots. We write $\mathfrak{q}=\mathfrak{q}_+=\mathfrak{k} + \mathfrak{p}_+$ and $\mathfrak{q}_-=\mathfrak{k} + \mathfrak{p}_-$. By definition of $\mathfrak{p}_+$ and $\mathfrak{p}_-$ as weight spaces, $\mathfrak{q}_+$ and $\mathfrak{q}_-$ are subalgebras of $\mathfrak{g}$.

We denote the set of positive roots by $\Delta^+$, and write $\Delta_c$ for the compact roots, which are the roots coming from $\mathfrak{k}$. The set of non-compact roots is defined as $\Delta_n=\Delta-\Delta_c$. By intersecting $\Delta^+$,can we define the positive compact roots $\Delta_c^+$ and the positive non-compact roots $\Delta_n^+$.

Finally, we will write $\mathfrak{U}(\mathfrak{g})$ for the universal enveloping algebra of $\mathfrak{g}$. For a weight $\lambda$ of $\mathfrak{g}$, $F_{\lambda}$ is the irreducible $\mathfrak{k}$-module with highest weight $\lambda$, and $E_{\lambda}$ is the irreducible $\mathfrak{g}$-module with highest weight $\lambda$. We use $N(\lambda)$ to denote the generalized Verma module $\mathfrak{U}(\mathfrak{g})\otimes_{\mathfrak{U}(\mathfrak{q})}F_{\lambda}$, which is a $\mathfrak{U}(\mathfrak{g})$-module .

\subsection{Irreducibility criterion}

For any $\alpha \in \Delta$ and $\lambda \in \mathfrak{t}^*$, we write $(\lambda)_{\alpha}=\frac{2<\lambda,\alpha>}{<\alpha,\alpha>}$  The half sum of the positive roots is written $\rho$, and we use $s_{\alpha}$ for the reflection through the hyperplane determined by the root $\alpha$. The following result about the irreducibility of $N(\lambda)$ appears as Corollary 6.3 and Theorem 6.4 in \cite{enright}, and the first part is originally due to Jantzen.

\begin{proposition}\label{irreducibility criterion} 
Assume for any $\alpha \in \Delta_n^+$ with $(\lambda+\rho)_{\alpha} \in \mathbb{Z}_{>0}$, there is $\gamma \in \Delta_n$ with
$(\lambda+\rho)_{\gamma}=0$ and $s_{\alpha}(\gamma) \in \Delta_c$. Then 
$N(\lambda)=\mathfrak{U}(\mathfrak{g})\otimes_{\mathfrak{U}(\mathfrak{q})}F_{\lambda}$ is irreducible. Moreover, if $\mathfrak{g}$ is of type $A_n$, it is both a necessary
and sufficient condition.
\end{proposition}

\subsection{\texorpdfstring{$(\mathfrak{g},K)$}{(g,K)}-modules}

To stay in an algebraic setting, this work takes place in the category of $(\mathfrak{g},K)$-modules, defined below. It allows us to us $K$ to denote a maximal compact subgroup in $G$ or its two-fold cover, as this distinction does not affect $(\mathfrak{g},K)$-modules.

\begin{definition}
A $(\mathfrak{g},K)$-module is a complex vector space $V$ with an action of $\mathfrak{g}$ and an action of $K$ such that 
\begin{enumerate}
    \item for all $v\in V, k \in K, X \in \mathfrak{g}$, we have $k\cdot (X \cdot v)=(Ad(k)X)\cdot(k\cdot v)$, 
    \item $V$ is $K$-finite, i.e., for every $v\in V$, the space generated by $K\cdot v$ is a finite-dimensional vector space,
    \item for all $v \in V, Y \in \mathfrak{k}$, we have $(\dfrac{d}{dt}\exp(tY)\cdot v)\mid_{t=0}=Y\cdot v$.
\end{enumerate}
\end{definition}
We recall the Frobenius reciprocity, together with one important corollary. 

\begin{proposition}[Frobenius reciprocity] 
Let $A$, $B$ be two rings with $A \subset B$. Let $M$ be an $A$-module and $N$ be a $B$-module. We have a vector space isomorphism $\textup{Hom}_B(B\otimes_A M,N)\cong\textup{Hom}_A(M,N).$
\end{proposition}

\begin{corollary}\label{projective}
Let $Q$ be an $A$-module, and let $P=B \otimes_A Q.$ If $Q$ is a projective $A$-module, then $P$ is a projective $B$-module.
\end{corollary}
As a consequence, we get the following result for $(\mathfrak{g},K)$-modules:
\begin{proposition}
Let $V$ be a $(\mathfrak{k},K)$-module. Then $\mathfrak{U}(\mathfrak{g})\otimes_{\mathfrak{U}(\mathfrak{k})}V$ is a projective $(\mathfrak{g},K)$-module.
\end{proposition}

\begin{proof}
By $K$-finiteness, every $(\mathfrak{k},K)$-module is projective as a $(\mathfrak{k},K)$-module. Now the result is a direct application of corollary \ref{projective} restricted to $(\mathfrak{g},K)$-modules.
\end{proof}

\subsection{Oscillator representation}

We are interested in a particular representation $\widetilde{\omega}$ of the metaplectic group $\widetilde{Sp}(2N, \mathbb{R})$, a double cover of the symplectic group. This representation, called oscillator representation, was first introduced in the $1960$s by Segal and Shale, in \cite{segal} and \cite{shale}, followed by the work of Weil in \cite{weil}. Several constructions and different models for the oscillator representation appear in \cite{li} and \cite{adams}.

For a subgroup $G$ of $Sp(2N, \mathbb{R})$, we denote by $\widetilde{G}$ its preimage in $\widetilde{Sp}(2N, \mathbb{R})$. We are only interested in algebraic $\widetilde{G}$-modules, hence we consider the category of $(\mathfrak{g},K)$-modules, for $K$ a maximal compact subgroup of $G$, or its two-fold cover. Therefore, we work with the Fock model of the oscillator representation, a realization of $\widetilde{\omega}$ as an $(\mathfrak{sp}(2N,\mathbb{C}),\widetilde{U}(N))$-module. We still call it the oscillator representation but denote it by $\omega$.

Restricting $\omega$ to $K$ gives a description of the $K$-types of the oscillator representation. This is done by using the duality correspondence detailed in section \ref{DualityCorrespondence} for the pair $\big(\widetilde{U}(1), \widetilde{U}(N)\big)$ using proposition \ref{TheraCorrU}. We obtain a list of $K$-types indexed by the weights $(k + \frac{1}{2},\dots,\frac{1}{2})$ of $\widetilde{U}(N)$, where $k = 0, 1, 2, \dots$. When $k$ is even, the $K$-types form one of the irreducible summand of $\omega$, as the $K$-types with $k$ odd form the other. Hence, $\omega$ is not irreducible but consists of two summands.

Since most of this work is done on the Lie algebra level, double covers do not play an important role. It is therefore enough to analyze subgroups $G,G'$ in a symplectic group, and it is not necessary to focus on their preimage $\widetilde{G},\widetilde{G}'$ in $\widetilde{Sp}(2N, \mathbb{R})$.

\subsection{Reductive dual pairs}

To decompose the oscillator representation restricted to a subgroup, we use dual pairs, following Howe's approach.

\begin{definition}
A pair $(G,G')$ of subgroups in a symplectic group $Sp(2N,\mathbb{R})$ is a reductive dual pair if
\begin{enumerate}
\item $G$ and $G'$ act reductively on $\mathbb{R}^{2N}$,
\item $G$ and $G'$ are centralizers of each other inside $Sp(2N,\mathbb{R})$.
\end{enumerate}
Moreover, if $G$ is compact, we say that $(G,G')$ is a compact dual pair.
\end{definition}

We assume that $G'$ is the smaller member of the pair so that the duality correspondence holds. We also consider two dual pairs with a particular relation, as introduced by Kudla in \cite{kudla}:

\begin{definition}
Two dual pairs $(G,G')$ and $(H,H')$ form a seesaw dual pair if we have the inclusions $H\subset G$ and $G'\subset H'$. We 
denote it by $\Big((G,G'),(H,H')\Big)$.
\end{definition}

Irreducible dual pairs, i.e., pairs that cannot be decomposed as a direct sum of two dual pairs, are classified in two types. Following \cite{howeTransc}, each pair correspond to either a division algebra (type II) or a division algebra with an involution (type I). We focus on pairs of type I, which come in four different types:
\begin{enumerate}
  \item $\left(O(p,q),Sp(2n,\mathbb{R})\right)$ , corresponding to $\mathbb{R}$ with the identity map,
  \item $\left(O(p,\mathbb{C}),Sp(2n,\mathbb{C})\right)$, corresponding to $\mathbb{C}$ with the identity map,
  \item\label{sym} $\left(U(r,s),U(p,q)\right)$, corresponding to $\mathbb{C}$ with the conjugation map,
  \item $\left(Sp(p,q),O^*(2n)\right)$, corresponding to $\mathbb{H}$ with the conjugation map.
\end{enumerate}
This corresponds to seven different cases, depending which group of the pair is the smallest (except for (\ref{sym}), which is symmetric).

\subsection{Duality correspondence}\label{DualityCorrespondence}

For a compact dual pair $(G,G')$, we decompose the oscillator representation $\omega$ of $Sp(2N,\mathbb{R})$ under the action of $G$. We obtain \[\omega=\bigoplus_{\sigma}(\Hom_G(\sigma,\omega)\otimes \sigma),\] summing over all the irreducible representations $\sigma$ of $G$. Indeed, if $T \in \Hom_G(\sigma, \omega)$ and $v \in \sigma$, then $T(v)\in\omega$ and we have a map $\Hom_G(\sigma,\omega)\times \sigma \to \omega$, $(T,v)\mapsto T(v).$ This map is extended to $\Hom_G(\sigma,\omega)\otimes\sigma \to \omega,$ which is injective when $\sigma$ is irreducible. Since $G$ is compact, $\omega$ is completely reducible, and $\omega=\bigoplus_{\sigma}(\Hom_G(\sigma,\omega)\otimes \sigma)$.

The duality correspondence gives an explicit description of $\theta(\sigma)$. By compactness of $G$, $\theta(\sigma)$ is a highest weight module, and we denote its highest weight by $\tau$. We write $E_{\tau}$ for the irreducible $\mathfrak{g}'$-module with highest weight $\tau$. Note that $\tau$ is also a dominant weight for $\mathfrak{k}'$, so $\tau$ is also the highest weight of a finite dimensional representation of $\mathfrak{k}'$. We use $F_{\tau}$ for the irreducible $\mathfrak{k}'$-module with highest weight $\tau$. We list now the duality correspondence for the pairs of type I. The correspondence, originally due to Kashiwara and Vergne \cite{vergne}, can be found with more details in \cite{adams}. 

\subsubsection{\texorpdfstring{$(O(n,\mathbb{R}),Sp(2p,\mathbb{R}))$}{(O(n,R),SP(2p,R))}}
Since $O(n,\mathbb{R})$ is a disconnected group, we use the embedding $O(n,\mathbb{R})=U(n)\cap GL(n,\mathbb{R})$. Given a highest weight $\lambda$ of $U(n)$ and a parameter $\epsilon=\pm1$, the representation of $O(n,\mathbb{R})$ with highest weight $(\lambda;\epsilon)$ is the irreducible summand of the representation of $U(n)$ with highest weight $\lambda$ containing the highest weight vector, tensored with the $sgn$ representation of $O(n,\mathbb{R})$ if $\epsilon=-1$.

\begin{proposition}\label{ThetaCorrO}
The duality correspondence for the pair $(O(n,\mathbb{R}),Sp(2p,\mathbb{R}))$ is given by the map $\sigma \mapsto \tau'$ described below: 
\[\left(a_1,\dots,a_k,0,\dots,0;\epsilon\right) \mapsto \left(a_1+\frac{n}{2},\dots,a_k+\frac{n}{2},\overbrace{\frac{n}{2}+1,\dots,\frac{n}{2}+1}^{\frac{1-\epsilon}{2}(n-2k)},
\frac{n}{2},\dots,\frac{n}{2}\right),\]
where $\sigma$ defines an irreducible highest weight $O(n,\mathbb{R})$-module and $\tau'$ defines an 
irreducible lowest weight $\mathfrak{sp}(2p,\mathbb{C})$-module. All such weights occur, with the constraints $k \leq [\dfrac{n}{2}]$,
and $k+\dfrac{1-\epsilon}{2}(n-2k) \leq p$.
\end{proposition}

\subsubsection{\texorpdfstring{$(U(p),U(m,n))$}{(U(p),U(m,n))}}
The duality correspondence for $(U(p),U(m,n))$ is expressed with $U(p)$-modules and $\mathfrak{gl}(m+n,\mathbb{C})$-modules:

\begin{proposition}\label{TheraCorrU}
The duality correspondence for the pair $(U(p),U(m,n))$ is given by the map $\sigma \mapsto \tau'$ described below:
\[\left(a_1+\frac{m-n}{2},\dots,a_k+\frac{m-n}{2},\frac{m-n}{2},\dots,\frac{m-n}{2},b_1+\frac{m-n}{2},\dots,b_l+\frac{m-n}{2}\right)\]
\[\mapsto\]
\[\left(a_1+\frac{p}{2},\dots,a_k+\frac{p}{2},\frac{p}{2},\dots,\frac{p}{2}\right)\oplus\left(-\frac{p}{2},\dots,-\frac{p}{2},b_1-
\frac{p}{2},\dots,b_l-\frac{p}{2}\right),\]
where $\sigma$ defines an irreducible highest weight $U(p)$-module and $\tau'$ defines an irreducible lowest weight 
$\mathfrak{gl}(m+n,\mathbb{C})$-module. All such 
weights occur, with the constraints $k+l \leq p$, $k\leq m$, $l \leq n$.
\end{proposition}

\subsubsection{\texorpdfstring{$(Sp(p),O^*(2n))$}{(Sp(p),O*(2n))}}
We recall that $Sp(p)$ can be seen either as the unitary quaternionic group, or as the intersection of $Sp(2p, \mathbb{C})$ and $U(2p)$. Its complexified Lie algebra is given by $\mathfrak{sp}(2p,\mathbb{C})$. The group $O^*(2n)=SO^*(2n)$ is the quaternionic orthogonal group. Its complexified Lie algebra is $\mathfrak{o}(2n,\mathbb{C})$.

\begin{proposition}\label{TheraCorrOstar}
The duality correspondence for the pair $(Sp(p),O^*(2n))$ is given by the map $\sigma \mapsto \tau'$ described below:
\[\left(a_1,\dots,a_k,0,\dots,0\right) \mapsto \left(a_1+p,\dots,a_k+p,p,\dots,p\right),\]
where $\sigma$ defines an irreducible highest weight $Sp(p)$-module and $\tau'$ defines an irreducible lowest weight 
$\mathfrak{o}(2n,\mathbb{C})$-module. All such weights occur, with the constraints $k \leq p$, $k\leq n$.
\end{proposition}

\section{Set-up and method}\label{method}
This section defines the notations, for $G$, $G'$ subgroups of a large symplectic group: 
\begin{itemize}
    \item $G,G'$ are real Lie groups, with complexified Lie algebras $\mathfrak{g}$, $\mathfrak{g}'$, forming a dual pair $(G,G')$ with $G'$ the smaller member
    \item $K,K'$ are maximal compact subgroups of $G, G'$ (or their two-fold covers), respectively, with complexified Lie algebras $\mathfrak{k}$, $\mathfrak{k}'$, and Cartan decomposition $\mathfrak{g}'=\mathfrak{k}'+\mathfrak{r}'$ for $\mathfrak{g}'$,
    \item $M'$ is the centralizer of $K$ so that $\big((K,M'),(G,G')\big)$ is a seesaw dual pair, with complexified Lie algebra $\mathfrak{m}'$ 
    \item $J'$ is a maximal compact subgroup of $M'$ (or its two-fold cover) with complexified Lie algebra $\mathfrak{j}'$, and Cartan decomposition $\mathfrak{m}'=\mathfrak{j}'+\mathfrak{p}'=\mathfrak{j}'+\mathfrak{p}'_++\mathfrak{p}'_-$,
    \item $\mathfrak{t'}$ is a Cartan subalgebra of both $\mathfrak{m}'$ and $\mathfrak{j}'$,
    \item $\mathfrak{q}'=\mathfrak{q}'_+=\mathfrak{j}'+\mathfrak{p}'_+$ and $\mathfrak{q}'_-=\mathfrak{j}'+\mathfrak{p}'_-$ are two parabolic subalgebras of $\mathfrak{m}'$.
\end{itemize}

To understand the restriction of $\omega$ to $G'$, we encounter two different cases. 
\begin{enumerate}
  \item $G'\subsetneq M'$ and $\mathfrak{r}'\cong\mathfrak{p}'_+$: we let $K$ act to get a decomposition $\displaystyle \omega=\oplus_{\sigma}(\sigma \otimes E_{\tau}),$ for $\sigma$ an irreducible representation of $K$ and $E_{\tau}$ an irreducible representation of $M'$ with highest weight $\tau$. We compute a condition (*) so that $N(\tau)$ is irreducible, which forces $E_{\tau}=N(\tau)=\mathfrak{U}(\mathfrak{m}')\otimes_{\mathfrak{U}(\mathfrak{q}')}F_{\tau}$, a projective $(\mathfrak{m}',J')$-module. The restriction from $M'$ to $G'$ is computed using $\mathfrak{r}'\cong\mathfrak{p}'_+$ to get \[\omega\cong\bigoplus_{\sigma}\left(\sigma \otimes \left(\mathfrak{U}(\mathfrak{g}')\otimes_{\mathfrak{U}(\mathfrak{k}')}(F_{\tau}\mid_{\mathfrak{k}'})\right)\right),\] where each summand is a projective $(\mathfrak{g}',K')$-module, under the condition (*).
  \item $(K,M')=(K_1,G')\oplus(K_2,G')$ with $K_1$ and $K_2$ of the same type: First, the action of $K_1 \times K_2$ decomposes $\omega=\omega_1\otimes \omega_2^*$, with $\omega_1$ a highest weight module for $K_1$, $\omega_2$ a lowest weight module for $K_2$. Each piece is decomposed as above. We compute conditions (*) and (**) so that $\omega_1$ and $\omega_2^*$, respectively, are projective $(\mathfrak{g}',K')$-modules. The tensor product is computed, so that when both conditions (*) and (**) are met, we have
  \[\omega=\omega_1\otimes \omega_2^*=\bigoplus_{\sigma, \widetilde{\sigma}}\left((\sigma\otimes\widetilde{\sigma})\otimes\left(\mathfrak{U}(\mathfrak{g}')\otimes_{\mathfrak{U}(\mathfrak{k}')}(F_{\tau}\otimes F_{\widetilde{\tau }})\right)\right),\] and each summand is a projective $(\mathfrak{g}',K')$-module. 
\end{enumerate}

Section \ref{irred} explores conditions so that $N(\tau)$ is irreducible, for each compact dual pair. Section \ref{identification} analyzes the restriction from $M'$ to $G'$ in the first case, and the tensor product in the second case. Finally, the results are summarized in section \ref{conclusion}. This work follows a strategy from Howe presented in \cite{howe}, using seesaw pairs to reduce the problem to unitary highest weight modules. In that work, Howe gives a similar result (theorem 5.2, \cite{howe}) but from an $L^2$ perspective.

\section{Irreducibility of \texorpdfstring{$N(\tau)$}{N(tau)}}\label{irred}

For each compact dual pair, we give a condition on the respective sizes of the groups so that the generalized Verma module $N(\tau)$ is irreducible. The stable range case is already known, see \cite{zhu}, for example, but our results show that this irreducibility holds in one more case. We also show that our bound cannot be extended in a general case, by giving counter-examples. 

\subsection{Dual pair \texorpdfstring{$(K,M')=(O(n,\mathbb{R}),Sp(2p,\mathbb{R}))$}{(K,M')=(O(n,R),Sp(2p,R))}} 

The group $M'=Sp(2p,\mathbb{R})$ has a maximal compact subgroup $J'=U(p)$. We have a correspondence between the highest weight $\sigma$ for $O(n,\mathbb{R})$ and the lowest weight $\tau'$ for $\mathfrak{sp}(2p,\mathbb{C})$,  in proposition \ref{ThetaCorrO}. We conjugate by the longest element of the Weyl group, switching positive and negative roots. Instead of the lowest weight $\tau'$ we work with the highest weight 
\[\tau=\left(-\dfrac{n}{2},\dots,-\dfrac{n}{2},\overbrace{-\dfrac{n}{2}-1,\dots,-\dfrac{n}{2}-1}^{\frac{1-\epsilon}{2}(n-2k)},-a_k-
\dfrac{n}{2},\dots,-a_1-\dfrac{n}{2}\right).\] Since we start with a highest weight $\sigma$ for $O(n,\mathbb{R})$, we have $a_1\geq \dots \geq a_k \geq 0$. The root system occurring here is given by:
\begin{itemize}
    \item $\Delta^+=\{e_i-e_j \mid 1\leq i < j\leq p\} \cup \{e_i+e_j \mid 1\leq i<j\leq p\}\cup\{2e_i \mid 1\leq i\leq p\}$
    \item $\Delta_n^+=\{e_i+e_j \mid 1\leq i<j\leq p\}\cup\{2e_i \mid 1\leq i\leq p\}$
    \item $\displaystyle\rho=(p,\dots,\overbrace{p+1-i}^{i\text{-th coordinate}},\dots,1).$
\end{itemize}

\subsubsection{Case $\epsilon=1$}

The products between $\tau+\rho$ and a non-compact positive root are:
\[(\tau+\rho)_{2e_i}= \begin{cases} 
p+1-i-\frac{n}{2} & \text{if } 1\leq i \leq p-k \\ 
p+1-i-\frac{n}{2}-a_{p+1-i} & \text{if } p-k < i \leq p 
\end{cases},\] 
\[(\tau+\rho)_{e_i+e_j}=\begin{cases}
2p+2-i-j-n & \text{if }1\leq i,j \leq p-k \\
2p+2-i-j-n-a_{p+1-j} & \text{if } 1\leq i \leq p-k < j \leq p \\
2p+2-i-j-n-a_{p+1-i}-a_{p+1-j} & \text{if } p-k < i,j \leq p
\end{cases}.\] 

If we take $n\geq 2p$, all these products are non-positive, and by proposition \ref{irreducibility criterion} $N(\tau)$ is irreducible. For $n=2p-1$, we see that $p+1-i-\frac{n}{2}=\frac{3}{2}-i$ is not an integer. All the other products are non-positive, so the criterion applies, and $N(\tau)$ is irreducible.

\subsubsection{Case $\epsilon=-1$}

The products of $\tau+\rho$ with non-compact positive roots are:

\[(\tau+\rho)_{2e_i}=\begin{cases}
p+1-i-\frac{n}{2} & \text{if } 1\leq i \leq p+k-n \\
p-i-\frac{n}{2} & \text{if } p+k-n < i \leq p-k \\
p+1-i-\frac{n}{2}-a_{p+1-i} & \text{if } p-k < i \leq p
\end{cases},\] 

\[(\tau+\rho)_{e_i+e_j}=\begin{cases}
2p+2-i-j-n & \text{if } 1\leq i,j \leq p+k-n \\
2p-i-j-n & \text{if } p+k-n < i,j \leq p-k \\
2p+2-i-j-n-a_{p+1-i}-a_{p+1-j} & \text{if } p-k < i,j \leq p \\
2p+1-i-j-n & \text{if } 1\leq i \leq p+k-n \\
& \text{and } p+k-n < j \leq p-k \\
2p+1-i-j-n-a_{p+1-j} & \text{if } p+k-n < i \leq p-k \\
& \text{and } p-k < j \leq p \\
2p+2-i-j-n-a_{p+1-j} & \text{if } 1\leq i \leq p+k-n\\
& \text{and } p-k < j \leq p
\end{cases}.\] 

For $n\geq 2p$, all these products are non-positive, hence $N(\tau)$ is irreducible. For $n=2p-1$, we have $p+1-i-\frac{n}{2}=\frac{3}{2}-i$ which is not an integer, and all the other products are non positive. We proved:

\begin{lemma}
If $n \geq 2p-1$, then the $(\mathfrak{sp}(2p,\mathbb{C}),\widetilde{U}(p))$-module $N(\tau)$ is irreducible for all the weights $\tau$ appearing in the restriction $\omega|_{\mathfrak{sp}(2p,\mathbb{C})}$.
\end{lemma}

If $n=2p-2$, we use theorem 6.2 from \cite{joseph}. Starting from $\sigma=(\overbrace{2,\dots,2}^{p-1},\overbrace{0,\dots,0}^{p-1};1)$, we get the highest weight $\tau=(-p+1,-p-1,\dots,-p-1)$ that we write as $\tau=(2,0,\dots,0)+(-p-1,\dots,-p-1)=(2,0,\dots,0)+(-p-1)\omega_{\alpha}$ following \cite{joseph}. For the family $N\left(u\omega_{\alpha}+(2,0,\dots,0)\right)$, the first reduction point of the family is given by $u=-p-1$. So $N(\tau)$ is reducible.

\subsection{Dual pair \texorpdfstring{$(K,M')=(U(p),U(m,n))$}{(K,M')=(U(p),U(m,n))}} 
We use the correspondence given in proposition \ref{TheraCorrU}. Conjugation by the longest element of the Weyl group of $M'$ sends $U(m,n)$ to $U(n,m)$ and switches positive and negative roots. To avoid confusion in the notation, we still denote our group by $M'$ after conjugation. We assume that $a_i$ and $b_j$ can be equal to zero, and we write the highest weight $\tau$ as 
\[\tau=\left(b_1-\dfrac{p}{2},\dots,b_n-\dfrac{p}{2}\right)\oplus\left(a_{n+1}+\dfrac{p}{2},\dots,a_{n+m}+\dfrac{p}{2}\right)\] with $b_n\leq \dots \leq b_1 \leq 0$ and $0 \leq a_{n+m} \leq \dots \leq a_{n+1}$.

We apply the irreducibility criterion given by proposition \ref{irreducibility criterion}. Our group $M'=U(n,m)$ contains a maximal compact subgroup $J'=U(n)\times U(m)$. The root system is of type $A_n$, hence this criterion is both necessary and sufficient for the irreducibility of $N(\tau)$. We have: 
\begin{itemize}
    \item $\Delta^+=\{e_i-e_j \mid 1\leq i<j\leq n+m\}$
    \item $\Delta_n^+=\{e_i-e_j \mid 1\leq i \leq n < j \leq n+m\}$
    \item $\rho=(\dfrac{m+n-1}{2},\dots,\overbrace{\dfrac{m+n-2i+1}{2}}^{i\text{-th coordinate}},\dots,\dfrac{-m-n+1}{2}).$
\end{itemize}

We obtain $(\tau+\rho)_{e_i-e_j}=b_i-a_j+j-i-p$ with $1\leq i \leq n < j \leq n+m$. Since $b_i-a_j\leq0$ for all $i,j$, we conclude that if $p \geq m+n-1$, then $(\tau+\rho)_{e_i-e_j}$ is non-positive for all 
$i,j$, and $N(\tau)$ is irreducible. We deduce:
\begin{lemma}
If $p \geq m+n-1$, then the $(\mathfrak{gl}(m+n,\mathbb{C}),\widetilde{U}(m)\times \widetilde{U}(n))$-module $N(\tau)$ is irreducible for all the weights $\tau$ appearing in the restriction $\omega|_{\mathfrak{gl_{m+n}(\mathbb{C})}}$.
\end{lemma}

If $p=m+n-2$, with $m,n\geq 2$, and \[\sigma=(\overbrace{1+\frac{m-n}{2},\dots,1+\frac{m-n}{2}}^{n-1},\overbrace{-1+\frac{m-n}{2},\dots,-1+\frac{m-n}{2}}^{m-1})\] we find the corresponding highest weight
\[\tau=\left(-\dfrac{p}{2},-1-\dfrac{p}{2},\dots-1-\dfrac{p}{2}\right)\oplus\left(1+\dfrac{p}{2},\dots,1+\dfrac{p}{2},\dfrac{p}{2}\right).\]
The products $(\tau+\rho)_{e_i-e_j}$ are strictly negative, except for $(\tau+\rho)_{e_1-e_{n+m}}=1$. But there is no root $\gamma$ such that $(\tau+\rho)_{\gamma}=0$. Since proposition \ref{irreducibility criterion} becomes a necessary condition for type $A_n$, this $N(\tau)$ is reducible.

\subsection{Dual pair \texorpdfstring{$(K,M')=(Sp(p),O^*(2n))$}{(K,M')=(Sp(p),O*(2n))}} 
Proposition \ref{TheraCorrOstar} gives the a correspondence between the highest weight $\sigma$ for $Sp(p)$ and the lowest weight $\tau'$ for $\mathfrak{o}(2n,\mathbb{C})$. As before, we conjugate by the longest Weyl group element, switch positive and negative roots, and obtain a highest weight $\tau$. Letting $a_i=0$ if necessary, we get 
\[\tau=\left(-p-a_n,\dots,-p-a_{n-i+1},\dots,-p-a_1\right)\] with $a_1\geq \dots \geq a_n \geq 0$.

A maximal compact subgroup of $M'$ is $J'=U(n)$. The complexified Lie algebra of $M'$ is of type $D_n$, therefore the root system of $M'$ is given by: 
\begin{itemize}
    \item $\Delta^+=\{e_i-e_j \mid 1\leq i<j\leq n\}\cup \{e_i+e_j \mid 1\leq i<j\leq n\}$
    \item $\Delta_n^+=\{e_i+e_j \mid 1\leq i < j \leq n\}$
    \item $\rho=(n-1,\dots,\overbrace{n-i}^{i\text{-th coordinate}},\dots,0).$
\end{itemize}

To apply proposition \ref{irreducibility criterion}, we calculate $(\tau+\rho)_{e_i+e_j}=2n-2p-i-j-a_{n-i+1}-a_{n-j+1}$ with $1\leq i <j \leq n$. For all $i,j$, we know that $-a_{n-i+1}-a_{n-j+1}\leq0$. So we conclude that if $p \geq n-\frac{3}{2}$, then $(\tau+\rho)_{e_i-e_j}$ is non-positive for all 
$i,j$, and $N(\tau)$ is irreducible. Since we only consider integral values of $n$ and $p$, we rewrite the bound as $p \geq n-1$. We proved:
\begin{lemma}
If $p \geq n-1$, then the $(\mathfrak{o}(2n,\mathbb{C}),\widetilde{U}(n))$-module $N(\tau)$ is irreducible for all the weights $\tau$ appearing in the restriction $\omega|_{\mathfrak{o}(2n,\mathbb{C})}$.
\end{lemma}

When $p=n-2$, we use theorem 6.2 from \cite{joseph} again to show that some modules $N(\tau)$ appearing in the restriction of $\omega$ are reducible. Choosing $\sigma=(1,\dots,1,0)$ gives a highest weight $\tau=(-p,-p-1,\dots,-p-1)=(-n+2,-n+1,\dots,-n+1)$, which is written as $(-n+1,\dots,-n+1)+(1,0,\dots,0)=(-n+1)\omega_{\alpha}+(1,0,\dots,0)$ following the notations from \cite{joseph}. The first reduction point of the family $N\left(u\omega_{\alpha}+(1,0,\dots,0)\right)$ is at $u=-n+1$, hence $N(\tau)$ with $\tau$ given above is reducible.

\section{Modules identifications}\label{identification}

The results presented in this section are known, see \cite{loke} when $(G,G')$ is in the table range, for example. For readability and consistency of notations, we still include our approach in this paper.

\subsection{Restriction from \texorpdfstring{$M'$}{M'} to \texorpdfstring{$G'$}{G'}} 

We start from $M'$, with complexified Lie algebra $\mathfrak{m}'$ and maximal compact subgroup $J'$. We recall the Cartan decomposition $\mathfrak{m}'=\mathfrak{j}'+\mathfrak{p}'$, with $\mathfrak{p}'=\mathfrak{p}'_++\mathfrak{p}'_-$, and we write $\mathfrak{q}'$ for $\mathfrak{j}'+\mathfrak{p}'_+$. The group $G'$ is a subgroup of $M'$, with complexified Lie algebra $\mathfrak{g}'$. We have a maximal compact subgroup $K'$ of $G'$, and the Cartan decomposition $\mathfrak{g}'=\mathfrak{k}'+\mathfrak{r}'$. 

We consider a finite-dimensional $\mathfrak{j}'$-module $E$, so $E$ is a $(\mathfrak{j}',J')$-module. By letting $\mathfrak{p}'_+$ act trivially, $E$ becomes a 
$\mathfrak{q}'$-module and we form $W=\mathfrak{U}(\mathfrak{m}')\otimes_{\mathfrak{U}(\mathfrak{q}')}E$, which is a $(\mathfrak{m}',J')$-module. We analyze the restriction of $W$ as a $(\mathfrak{g}',K')$-module. As vector spaces, we have $W\cong S(\mathfrak{p}'_+)\otimes E$, with $S(\mathfrak{p}'_+)$ the symmetric algebra on $\mathfrak{p}'_+$.
From $E$, we also create a $(\mathfrak{g}',K')$-module: by restriction, we see $E$ as a $\mathfrak{k}'$-module, $E\mid_{\mathfrak{k}'}$, and form
the tensor product $V=\mathfrak{U}(\mathfrak{g}')\otimes_{\mathfrak{U}(\mathfrak{k}')}(E\mid_{\mathfrak{k}'})$. Similarly, there is an isomorphism of vector spaces $V\cong 
S(\mathfrak{r}')\otimes (E\mid_{\mathfrak{k}'})$.

We define two filtrations, $V=\oplus_nV_n/V_{n-1}$ and $W=\oplus_nW_n/W_{n-1}$, by \[V_n=\sum_{r\leq
n}S(\mathfrak{r}')[r]\mathfrak{U}(\mathfrak{k}')\otimes_{\mathfrak{U}(\mathfrak{k}')}(E\mid_{\mathfrak{k}'}),\]and \[W_n=\sum_{r\leq n}S(\mathfrak{p}'_+)[r]\mathfrak{U}(\mathfrak{j}')\otimes_{\mathfrak{U}(\mathfrak{q}')}E.\] By Frobenius reciprocity, we have a map $T:V\to W, 1\otimes e \mapsto 1\otimes e$ for any $e\in E$. 
Writing $\{x_1,\dots,x_r\}$ for a basis of $\mathfrak{r}'$, $\{y_1,\dots,y_r\}$ for a basis of $\mathfrak{p}'_+$ and 
$\{z_1,\dots,z_r\}$ for a basis of $\mathfrak{p}'_-$ such that $x_i=y_i+z_i$ in $\mathfrak{p}'$, we extend the map $T$ linearly so that 
\[T(x_1\dots x_n \otimes e)=(y_1+z_1)\dots(y_n+z_n)\otimes e\] for any $e \in E$. The map $T:V\to W$ now preserves the filtrations, which proves:

\begin{lemma}\label{restrictionM'G'}
The map $T:V \to W$ is an isomorphism of $\mathfrak{U}(\mathfrak{g}')$-modules, induced by an isomorphism of $S(\mathfrak{r}')$-modules on the graded spaces $T_G:Gr(V)\to Gr(W)$, through the identification $\mathfrak{r}'\cong\mathfrak{p}'_+$.
\end{lemma}
This implies that
\[\Big(\mathfrak{U}(\mathfrak{m}')\otimes_{\mathfrak{U}(\mathfrak{q}')}E\Big)\mid_{\mathfrak{g}'}\cong\mathfrak{U}(\mathfrak{g}')\otimes_{\mathfrak{U}(\mathfrak{k}')}(E\mid_{\mathfrak{k}'}).\]

This is applied to $E=F_{\tau}$ in the cases where $\mathfrak{r}'$ and $\mathfrak{p}'_+$ are isomorphic.

\subsection{Tensor product for \texorpdfstring{$(K,M')=(K_1,G')\oplus(K_2,G')$}{(K,M')=(K1,G')+(K2,G')}}
We start from the Cartan decomposition $\mathfrak{g}'=\mathfrak{k}'+\mathfrak{r}'$. We decompose $\mathfrak{r}'$ further as $\mathfrak{r}'_++\mathfrak{r}'_-$ and note that $\mathfrak{r}'_+$ and $\mathfrak{r}'_-$ are commutative Lie algebras. We define $\mathfrak{q}'_+=\mathfrak{k}'+\mathfrak{r}'_+$ and $\mathfrak{q}'_-=\mathfrak{k}'+\mathfrak{r}'_-$.

We consider two finite-dimensional $\mathfrak{k}'$-modules $E$ and $F$; these are $(\mathfrak{k}',K')$-modules. We let $\mathfrak{r}'_+$ act on $E$ by zero, so $E$ 
becomes a $\mathfrak{q}'_+$-module. Similarly, we let $\mathfrak{r}'_-$ act on $F$ by zero and obtain a $\mathfrak{q}'_-$-module. We define
$V_E=\mathfrak{U}(\mathfrak{g}')\otimes_{\mathfrak{U}(\mathfrak{q}'_+)}E$ and $V_F=\mathfrak{U}(\mathfrak{g}')\otimes_{\mathfrak{U}(\mathfrak{q}'_-)}F,$
which are $(\mathfrak{g}',K')$-modules. We also define $V=\mathfrak{U}(\mathfrak{g}')\otimes_{\mathfrak{U}(\mathfrak{k}')}(E\otimes F)$.

By Poincaré-Birkhoff-Witt theorem, there exists a grading on both $\mathfrak{U}(\mathfrak{r}'_+)$ and $\mathfrak{U}(\mathfrak{r}'_-)$. Since $\mathfrak{r}'_+$ and $\mathfrak{r}'_-$ are commutative, $\mathfrak{U}(\mathfrak{r}'_+)=S(\mathfrak{r}'_+)$ and $\mathfrak{U}(\mathfrak{r}'_-)=S(\mathfrak{r}'_-)$. We identify the piece $\mathfrak{U}(\mathfrak{r}'_+)[n]$ of degree $n$ with the space $S(\mathfrak{r}'_+)[n]$ of homogeneous polynomials of degree $n$ (same for $\mathfrak{r}'_-$). We write $M_n$ for the subspace of elements of degree less or equal to $n$ in 
$\mathfrak{U}(\mathfrak{r}')$: \[M_n=\sum_{r+s\leq n}(S(\mathfrak{r}'_-)[r]\otimes S(\mathfrak{r}'_+)[s])\cong\oplus_{i\leq n}S(\mathfrak{r}')[i].\] From this, we define a filtration $V=\oplus_nV_n/V_{n-1}$ where 
$V_n=M_n\mathfrak{U}(\mathfrak{k}')\otimes_{\mathfrak{U}(\mathfrak{k}')}(E\otimes F)$. Note that $V_0=E\otimes F$. By the description of $M_n$ as $\oplus_{i\leq 
n}S(\mathfrak{p}')[i]$, the quotient $M_n/M_{n-1}$ is identified with $S(\mathfrak{r}')[n]$ and we get
$V_n/V_{n+1}=S(\mathfrak{r}')[n]\otimes_{\mathfrak{U}(\mathfrak{k}')} (E\otimes F).$
We define a similar filtration on $V_E\otimes V_F$: \[(V_E\otimes V_F)_n=\sum_{r+s\leq 
n}\left(\left(M_r\mathfrak{U}(\mathfrak{k})\otimes_{\mathfrak{U}(\mathfrak{q}'_+)}E\right)\otimes \left(M_s\mathfrak{U}(\mathfrak{k})\otimes_{\mathfrak{U}(\mathfrak{q}'_-)}F\right)\right).\] As vector spaces, this is equivalent to \[(V_E\otimes V_F)_n=
\sum_{r+s\leq n} \Big(S(\mathfrak{r}'_-)[r]\otimes_{\mathfrak{U}(\mathfrak{q}'_+)}E \Big)\otimes\Big(S(\mathfrak{r}'_+)[s]\otimes_{\mathfrak{U}(\mathfrak{q}'_-)}F\Big).\] Hence we have
\[(V_E\otimes V_F)_n/(V_E\otimes V_F)_{n-1}=\sum_{r+s= n} \Big(S(\mathfrak{r}'_-)[r]\otimes_{\mathfrak{U}(\mathfrak{q}'_+)}E 
\Big)\otimes\Big(S(\mathfrak{r}'_+)[s]\otimes_{\mathfrak{U}(\mathfrak{q}'_-)}F\Big).\]

Since $E\otimes F=V_0$ is naturally a subset of $V$, we use Frobenius reciprocity to extend this inclusion to a map 
\[T:V \to V_E\otimes V_F, 1\otimes (e\otimes f) \mapsto (1\otimes e)\otimes(1\otimes f),\] for all $e \in 
E$ and $f \in F$. We extend this map so that it is is compatible with the module structure. By using basis of $\mathfrak{r}'_+$, $\mathfrak{r}'_-$, it is a simple computation to show that $T$ preserves the filtrations.

\begin{lemma}\label{tensor}
The map $T:V \to V_E\otimes V_F$, induced by an isomorphism of $S(\mathfrak{r}')$-modules on 
the graded spaces $T_G:Gr(V)\to Gr(V_E\otimes V_F)$, is an isomorphism of $\mathfrak{U}(\mathfrak{g}')$-modules.
\end{lemma}

\begin{proof}
We know that $T(V_n)\subset (V_E\otimes V_F)_n$. Using computations with the action of basis elements of $\mathfrak{r}'_+$ and $\mathfrak{r}'_-$ and by tracking the degrees, we show that $T_G:Gr(V)\to Gr(V_E\otimes V_F)$ is surjective. Finally, as $\mathbb{C}$-vector spaces, we have 
\[\dim(V_n)=\dim(E)\dim(F)\Big(\sum_{r+s\leq n}\dim(S(\mathfrak{p}'_-)[r])\dim(S(\mathfrak{p}'_+)[s])\Big)=\dim((V_E\otimes V_F)_n),\] which concludes the proof.
\end{proof}

Hence, we proved that $V_E\otimes V_F\cong V$, i.e., \[\left(\mathfrak{U}(\mathfrak{g}')\otimes_{\mathfrak{U}(\mathfrak{q}'_+)}E\right) \otimes \left(\mathfrak{U}(\mathfrak{g}')\otimes_{\mathfrak{U}(\mathfrak{q}'_-)}F\right) \cong \mathfrak{U}(\mathfrak{g}')\otimes_{\mathfrak{U}(\mathfrak{k}')}(E\otimes F).\] 

\section{Conclusion}\label{conclusion}

\begin{theorem}\label{MainResult}
Let $G'$ be the smaller member of a dual pair $(G,G')$ in a symplectic group $Sp(V)$. Then the restriction of the Fock model of the oscillator representation $\omega$ of $\widetilde{Sp}(V)$ to $G'$ is a projective $(\mathfrak{g}',K')$-module under the condition (*) listed in the table below:
\begin{center}
\begin{tabular}{cccc}
&$(G,G')$ & $(K,M')$ &  (*)\\
\hline
\hline
(i) & $\left(Sp(2n,\mathbb{R}),O(p,q)\right)$ & $\left(U(n),U(p,q)\right)$ & $n \geq p+q-1$ \\
\hline
(ii) & $\left(O(p,q),Sp(2n,\mathbb{R})\right)$ & $\left(O(p),Sp(2n,\mathbb{R})\right)$ & $p,q \geq 2n-1$ \\
& & $\oplus \left(O(q),Sp(2n,\mathbb{R})\right)$ & \\
\hline
(iii) & $\left(O^*(2n),Sp(p,q)\right)$ & $\left(U(2n),U(2p,2q)\right)$ & $n \geq 2(p+q)-1$ \\
\hline
(iv) & $\left(Sp(p,q),O^*(2n)\right)$ & $\left(Sp(p),O^*(2n)\right)$ & $p,q \geq n-1$ \\
& & $\oplus \left(Sp(q),O^*(2n)\right)$ & \\
\hline
(v) & $\left(Sp(2n,\mathbb{C}),O(p,\mathbb{C})\right)$ & $\left(Sp(n),O^*(2p)\right)$ & $n \geq p-1$ \\
\hline
(vi) & $\left(O(p,\mathbb{C}),Sp(2n,\mathbb{C})\right)$ & $\left(O(p),Sp(4n,\mathbb{R})\right)$ & $p \geq 4n-1$ \\
\hline
(vii) & $\left(U(r,s),U(p,q)\right)$ & $\left(U(r),U(p,q)\right)$ & $r,s\geq p+q-1$ \\
& & $ \oplus \left(U(s),U(p,q)\right)$ & \\
\end{tabular}
\end{center}
\end{theorem}

\begin{proof}
For cases (i), (iii), (v) and (vi), the action of $K$ decomposes $\omega$ as $\displaystyle \omega=\oplus_{\sigma}(\sigma \otimes E_{\tau}).$ Section \ref{irred} shows that (*) is a necessary condition for the equality $N(\tau)=E_{\tau}$. Finally, the restriction from $M'$ to $G'$ is computed in \ref{restrictionM'G'}, which requires $\mathfrak{p}'_+ \cong \mathfrak{r}'$. This is verified in the table below. For complex groups like $Sp(2n,\mathbb{C})$ or $O(p,\mathbb{C})$, we consider their Lie algebras as real manifolds, before complexifying them.

\begin{center}
\begin{tabular}{c|ccc| ccc}
& $\mathfrak{m}'$ & $\mathfrak{j}'$ & $\mathfrak{p}'_+$ & $ \mathfrak{g}'$ &  $\mathfrak{k}'$ & $\mathfrak{r}'$ \\
\hline
\hline
(i) & $\mathfrak{gl}(p+q,\mathbb{C})$ & $\mathfrak{gl}(p,\mathbb{C})$ & $M_{p,q}(\mathbb{C})$ &  $\mathfrak{o}(p,q,\mathbb{C})$ & $\mathfrak{o}(p,\mathbb{C})$ & $M_{p,q}(\mathbb{C})$ \\
& & $\oplus \mathfrak{gl}(q,\mathbb{C})$ & & & $\oplus \mathfrak{o}(q,\mathbb{C})$ & \\
\hline
(iii) & $\mathfrak{gl}(2p+2q,\mathbb{C})$ & $\mathfrak{gl}(2p,\mathbb{C})$ &  $M_{2p,2q}(\mathbb{C})$ & $\mathfrak{sp}(p,q)$ & $\mathfrak{sp}(p)$ & $M_{2p,2q}(\mathbb{C})$ \\
& & $\oplus \mathfrak{gl}(2q,\mathbb{C})$ & & & $\oplus \mathfrak{sp}(q)$ & \\
\hline
(v) & $\mathfrak{o}(2p,\mathbb{C})$ & $\mathfrak{gl}(p,\mathbb{C})$ & $\mathfrak{o}(p,\mathbb{C})$ & $\mathfrak{o}(p,\mathbb{C})$ & $\mathfrak{o}(p,\mathbb{C})$ & $\mathfrak{o}(p,\mathbb{C})$ \\
& & & & $\oplus \mathfrak{o}(p,\mathbb{C})$ & & \\
(vi) & $\mathfrak{sp}(4n,\mathbb{C})$ & $\mathfrak{gl}(2n,\mathbb{C})$ & $\mathfrak{sp}(2n,\mathbb{C})$ & $\mathfrak{sp}(2n,\mathbb{C})$ & $\mathfrak{sp}(2n,\mathbb{C})$ & $\mathfrak{sp}(2n,\mathbb{C})$ \\ 
& & & & $\oplus\mathfrak{sp}(2n,\mathbb{C})$ & & \\
\end{tabular}
\end{center}

For cases (ii), (iv) and (vii), $\omega$ is decomposed under the action of $K_1 \times K_2$ as $\omega=\omega_1\otimes \omega_2^*$. The condition (*) is computed for each case to get $N(\tau)=E_{\tau}$ in section \ref{irred}. The two pieces $\omega_1$ and $\omega_2^*$ are put back together through the tensor product described in \ref{tensor}. To apply this result, we need to verify that $\mathfrak{r}'$ can be decomposed into the weight spaces $\mathfrak{r}'_+$ and $\mathfrak{r}'_-$. This corresponds to checking that $\mathfrak{k}'$ has a non-trivial center, and is done explicitly for each case:
\begin{itemize}
\item[(ii)] $G'=Sp(2n,\mathbb{R}), \mathfrak{g}'=\mathfrak{sp}(2n,\mathbb{C}), \mathfrak{k}'=\mathfrak{gl}(n,\mathbb{C})$,
\item[(iv)] $G'=O^*(2n), \mathfrak{g}'=\mathfrak{o}(2n,\mathbb{C}), \mathfrak{k}'=\mathfrak{gl}(n,\mathbb{C})$,
\item[(vii)] $G'=U(p,q), \mathfrak{g}'=\mathfrak{gl}(p+q,\mathbb{C}), \mathfrak{k}'=\mathfrak{gl}(p,\mathbb{C})\oplus \mathfrak{gl}(q,\mathbb{C})$,
\end{itemize}
which concludes the proof as $\mathfrak{gl}(n,\mathbb{C})$ has non-trivial center.
\end{proof}

\medskip

\bibliographystyle{acm}
\bibliography{oscillator}
\end{document}